\documentclass[11pt]{article}

\usepackage{pb-diagram}
\usepackage{times}
\usepackage{amsthm}
\usepackage{color}
\usepackage{amsmath}
\usepackage{amssymb}
\usepackage{mathrsfs}
\usepackage{comment}
\usepackage[english]{babel}
\usepackage{amsfonts}
\usepackage[pdftex]{graphicx}
\usepackage{anysize}
\usepackage{setspace}
\usepackage{textcomp}
\usepackage{cancel}
\newtheorem{Theorem}{Theorem}[section]
\newtheorem{Remark}[Theorem]{Remark}
\newtheorem{Lemma}[Theorem]{Lemma}

\newtheorem{Corollary}[Theorem]{Corollary}
\newtheorem{Proposition}[Theorem]{Proposition}
\newtheorem{Definition}[Theorem]{Definition}

\def\R{\mathbb{R}}

\def\d{\partial}

\def\de{\delta}

\def\a{\alpha}
\def\b{\beta}
\def\e{\varepsilon}

\def\hg{\hat{g}}
\def\bc{\bar{c}}
\def\bg{\bar{g}}
\def\ph{\phi}
\def\ps{u}

\def\lap{\Delta}
\def\hess{{\rm Hess}}
\def\th{\tilde{h}}

\def\hd{\hat{\partial}}

\begin{document}

\title{\bf Existence of CMC-foliations in asymptotically cuspidal manifolds}

\date{\today}

\author{Claudio Arezzo and Karen Corrales 
}
\date{\small}

\maketitle \vspace{-1cm}

\begin{abstract}
\noindent In this paper we prove existence and uniqueness of a CMC foliation in asymptotically cuspidal manifolds. Moreover, we study the isoperimetric problem in this case. Our proof does not require any curvature assumption and it holds for any dimension.
\end{abstract}

\section{Introduction}\label{intro}

\noindent One of the most classical questions in studying the Riemannian properties of a manifold $(M,g)$ is to describe 
its constant mean curvature (CMC) submanifolds. Besides some beautiful classical results inspired by the celebrated Alexandrov Theorem
\cite{Alex} in space forms, a general study of existence and uniqueness of foliations with (stable) CMC hypersurfaces has started in the nineties thanks to the work of Ye \cite{Ye}, who found sufficient conditions for the existence
of a CMC foliation by spheres in a small neighbourhood of a point, and by Huisken-Yau \cite{HuY} for the existence
of a CMC foliation by spheres at infinity of an asymptotically Schwarzschild end of a complete manifold.

\noindent These two works have motivated a huge amount of further research: as for Ye's Theorem, his sufficient conditions have been later weakened in a number of ways for example by \cite{MMP, N, PX}, while Huisken-Yau's work, seen as a first case of an asymptotically Euclidean manifold, has generated a (almost) complete study about existence and uniqueness of CMC foliations at infinity (and the isoperimetric properties of the leafs)
in the classical trichotomy in Riemannian geometry for which we collect a non-exhaustive but illustrative summary:

\begin{itemize}
\item
{\emph{Positive curvature:}}
The exact spherical case was studied by Brendle \cite{B} as a particular case of a general theory he developed for warped products manifolds. In this case he has shown the validity of a strong Alexandrov-type Theorem. Existence and uniqueness of a CMC foliation in the asymptotic case is, to be best of authors' knowledge, an open problem.

\item
{\emph{Zero curvature:}}
Because of its natural interest and its applications in General Relativity, this is by far the most studied case since the starting results by Huisken-Yau in the special case of asymptotically Schwarzschild case in dimension $3$. Beautiful generalizations to any dimension and to general asymptotically flat cases can be found in   \cite{Br, EM, EM2, H, Ne, QT}.

\item
{\emph{Negative curvature:}}
The analogous question in the asymptotically hyperbolic setting has started by Rigger \cite{R}, Neves-Tian \cite{NT, NT2} and Mazzeo-Pacard \cite{MP}, for the Anti-de Sitter-Schwarzchild asymptotics and the conformally compact case. More recently, Chodosh \cite{Ch} studied large isoperimetric regions in asymptotically hyperbolic manifolds, while Ambrozio \cite{Am} and Lancini \cite{L} managed to give a quantitative version of some of the previous results and to apply them provide a partial solution to the Penrose Inequality in this setting. In all these results the model space at infinity is taken to be of infinite volume.

\end{itemize}

\noindent The original motivation of this paper was to complete the above picture in the negative curvature case by analysing the  {\emph{finite volume}} situation, still extensively studied especially in connection with classical hyperbolic geometry and known  as {\emph{cusps}} in its literature. In doing so, it turned out that the answer to this problem does not depend on any curvature assumption but just on the requirement on the ``cuspidal" radial behaviour of the Riemannian metric at the ends:

 \begin{Definition}
 \label{ACdef}
 \begin{description}
 \item[(i)] Given a compact Riemannian $n$-manifold $(\Sigma,\bar{g})$, we say that  $$(N,g_N)=((0,\infty)\times\Sigma, dr^2+e^{-2r}\bg)$$ is a cusp with slice $(\Sigma,\bar{g})$.

\item[(ii)] A complete $(n+1)$-dimensional Riemannian manifold $(M,g)$ is called {\emph{$C^{k,\b}$-asymptotically cuspidal of order $\boldsymbol{\a} = (\a_1, \dots, \a_k), \,\a_l > 0 \,\, \forall \,l = 1,\dots, k,$}} if there exists a compact subset $K\subset M$ 
such that
\begin{itemize}
\item
$M\setminus K \simeq \coprod_{l=1}^k (0,\infty)\times \Sigma_l = \coprod _{l=1}^k N_l$,  
\item
For every $l$, the restriction of the diffeomorphism above
$\Phi_l \colon (0,\infty)\times\Sigma_l \rightarrow M\setminus K$ verifies \linebreak$\Phi_l^*g = 
g_{N_{l}}+h^{(l)}$,
	where $$|h^{(l)}_{ij}|_{C^{k,\b}(N_l)}\leq C e^{-r\a_l},\quad \mbox{for some $C>0$.}$$
\end{itemize}

\end{description}
\end{Definition}

\noindent It is immediate to observe that  $N$ has finite volume, the area of slices $\{r\}\times \Sigma$ 
decreases exponentially in $r$, and the scalar curvature $R_{g_N}$ of $N$ is given by $R_{g_N}=e^{2r}R_{\bg}-n(n+1)$, 
hence it is bounded only for scalar flat slices. Moreover
	each slice $\{r\}\times \Sigma$ has constant mean 
	curvature equal to $-n$ (with respect to the unit normal vector $\d_r$).

\noindent Our main results, Theorems \ref{ThCMCsurface}, \ref{ThCMCfoliation} and \ref{Thisoperimetric}, provide existence, local uniqueness and the isoperimetric property of a CMC foliation near infinity for any asymptotically cuspidal manifold. It can be summarized as follows

\medskip
\begin{Theorem} \label{Main}
Given   a $C^{2,\b}$-asymptotically cuspidal manifold $(M,g)$ with $\a _l >4, \forall \,l=1, \dots, k$. Then, for every $l$ holds:
\begin{description}
\item[(i)]
there exists $R_l>0$ such that for every $r>R_l$ there exists a constant $C(R_l)>0$ and a unique function $u^{(l)}_r \in C^{2,\b}(\Sigma_l)$ with zero mean value such that $|u^{(l)}_r|_{C^{2,\b}(\Sigma_l)}\leq Ce^{-r(\a_l-2)}$ and
$$S_r(u^{(l)}_r)=\{(r+u^{(l)}_r(x))\mid x\in \Sigma_l\}\subset N_l$$
has constant mean curvature, with respect to the metric $\Phi_l^*g$;
\item[(ii)]
having denoted  by $H(r,u_r^{(l)},g)$ to the mean curvature of $S_{r}(u^{(l)}_r)$ with respect to the metric $g$, we have
the mean curvature of the CMC perturbed slices above satisfies
$$H(r,u^{(l)}_r,\Phi_l^*g) = -n +c e^{-r(\a_l -4)},$$
for some constant $c$;
\item[(iii)]
there exists $\tilde{R_l}\geq R_l >0$ such that $\{S_{r}(u^{(l)}_r)\}_{r\in (\tilde{R_l},\infty)}$ gives a CMC foliation of a non-compact region of each cuspidal end of $(M,g)$. Moreover each leaf of this foliation is a stable critical point of the perimeter functional with volume constraint. 
\item[(iv)]
$\exists \epsilon_0 >0$ s.t. $\forall \epsilon <\epsilon_0$, the boundary of the isoperimetric region of volume $\epsilon$ is given by $S_{r}(u^{(l)}_r)$, for some $l$.
\end{description}
\end{Theorem}

\noindent It is interesting to put the above result also in light of a purely PDE's approach. All cases appearing in the Riemannian trichotomy mentioned above, fall into the category of {\emph{warped product}} metrics, \linebreak$g= dr^2+\phi(r)^2\bg$, and it is easy to see that the linearized mean curvature operator takes the form

	$$ L_{r}(u)=-\phi^{-2}(r)\left(\lap_{\bg}\ps+n\ps\left((\dot{\ph}(r))^2-\ph(r)\ddot{\ph}(r)\right)\right).$$

\noindent Of course failure of invertibility of this operator gives one measure of the difficulty of the problem studied in this paper. From this point of view, while at one of the spectrum we find Schwarzschild, Anti-de Sitter-Schwarzschild, and non-trivial cones in the sense of \cite{CEV}, where the above operator is actually invertible. At the other extreme, we find the general Euclidean case or the spherical one. In this perspective, it is immediate to see that the cuspidal case is actually the only case where the kernel of $L_r$ is non-zero and it is spanned by constant functions.

\begin{Remark}
It would be interesting to investigate the relationship between the cuspidal case here studied and the original Ye's problem for small neighbourhoods of Riemannian manifolds. It is in fact easy to observe that if $(\Sigma, \bar{g})$ is the round sphere then, under the change of coordinates $s=e^{-e^{r}}$, for $s\rightarrow 0$,
$$e^{2u_0}(g_N +h)  \sim (ds^2 + s^2\bar{g}) + \mathcal{O}\left(\frac{s^2}{\log s}\right), $$
if $|h_{ij}|_{C^{k,\b}(N)}\leq C e^{-r\a},$ where $u_0 = r-e^r$ and $\a>4$ as in our results. It seems quite intriguing to study the correspondence between Ye's foliation and ours via this conformal change.
\end{Remark} 

\noindent The organization of the paper is as follows: In Section \ref{notation}, we will collect basic facts, notations and properties of asymptotically cuspidal manifolds and mean curvature. Section \ref{existenceS} will be devoted to provide a proof of parts (i) and (ii) of Theorem \ref{Main}. 
 In Section \ref{foliation}, we give an alternative proof of part (i) which allows also to prove part (iii). Finally, Section \ref{isoperimetric} contains the proof of part (iv).
 
 \vspace{2mm}

\noindent \textbf{Aknowledgements:} During the preparation of this work we benefited from many beautiful and instructive discussions with
Francesco Maggi, Samuele Lancini and Lorenzo Mazzieri for which we are deeply grateful.

\section{Preliminaries}\label{notation}
Throughout this paper, we will consider $(M,g)$ as a $C^{2,\b}$-asymptotically cuspidal manifold of order $\a>4$ and as a common notation, the geometric quantities relative to $(\Sigma,\bg)$ will be over line.

All the proofs involve working on one end at each time, so from now on, we assume $l=1$, $\boldsymbol{\a}=(\a_1)$ and $\a_1 = \a$. For simplicity, we will refer to each end as $(N=(0,\infty)\times \Sigma,g=g_N+h).$
 

\begin{Definition}
	\label{slice}
Let $r_0>0$ and $u\in C^{2,\beta}(\Sigma)$ with $\b\in (0,1)$. We define the hypersurface 
$$S_{r_0}(u)=\{(r_0+u(x),x)\mid x\in\Sigma \}\subset N,$$
as the graph of $\ps$ over the slice $\Sigma_{r_0}:=\{r_0\}\times \Sigma.$
\end{Definition}

It is convenient to write $S_{r_0}(u)$ as a level set of a function and to find a suitable set of coordinates on $N$ adopted to the graph of $u$. So, we set
$$F:(0,\infty)\times \Sigma\to \R,\quad F(r,x)=r-(r_0+\ps(x)),\quad x=(x_1,\cdots ,x_n).$$ 

Straightforward computations on $F$ imply the following 
\begin{Proposition}
	\label{F}
	If $\d_k$ denotes the partial derivative with respect $x_k$, then
	\begin{description}
		\item[(i)] The vectors $\{\d_k \ps\d_r+\d_k\}$ are $g$-orthogonal to $\nabla F$.
		\item[(ii)] $|\nabla F|=\sqrt{1+e^{2r}|\overline{\nabla}\ps|^2_{\bg}+E^F},$	where $E^F=E^F(r,h,\d_k u,(\d_k u)^2).$  
		\item[(ii)] There exists a constant $\mathcal{C}_1$ such that $|E^F(u=0)|_{C^{0,\b}(\Sigma)}\leq\mathcal{C}_1 e^{-r_0\a}.$
	\end{description}
\end{Proposition}

Therefore, the seeked set of coordinates on $S_{r_0}(u)$ is given by  $\{\hd_F=|\nabla F|^{-2}\nabla F, \hd_k=\d_k \ps\d_ r+\d_k\}$ and the metric $g$ in this coordinates is given by
\begin{equation}
\begin{split}
\label{hg}
g_{FF}&:=g(\hd_F,\hd_F)=|\nabla F|^{-2}\\
g_{Fj}&:=g_{iF}=g\left(\hd_F,\hd_j\right)=g\left(\hd_i,\hd_F\right)=0\\
g_{ij}&:=g\left(\hd_i,\hd_j\right)=\d_i\ps\d_j\ps g_{00}+\d_i\ps g_{0j}+\d_j\ps g_{i0}+g_{ij}.
\end{split}
\end{equation}

Given a function $u\in C^{2,\beta}(\Sigma),$ we will denote by $H(r_0,u,g)$ to the mean curvature of $S_{r_0}(u)$ with respect to the metric $g$.

The following formula, inspired by \cite{L} for the hyperbolic case of infinite volume, will be crucial for the rest of the paper
\begin{Proposition}
		\label{Hg}
	Let $u\in C^{2,\beta}(\Sigma)$ be a positive function. Then the mean curvature $H(r_0,u,g)$ satisfies
	\begin{equation*}
	\begin{split}
	H(r_0,u,g)\sqrt{1+e^{2r}|\overline{\nabla}\ps|^2_{\bg}+E^F}\, =-e^{2r}\lap_{\bar{g}}\ps+e^{4r}\frac{\hess_{\bg}\ps(\overline{\nabla}\ps,\overline{\nabla}\ps)}{1+e^{2r}|\overline{\nabla}\ps|^2_{\bg}}-\frac{e^{2r}|\overline{\nabla}\ps|^2_{\bg}}{1+e^{2r}|\overline{\nabla}\ps|^2_{\bg}}-n+E^H,
	\end{split}
	\end{equation*}
	
	where $r=r_0+u(x)$ and $E^H=E^H(r,h,\d_k h,\d_k u, \d^2_{ij}u).$ 
\end{Proposition}

\begin{proof}
	In coordinates $\{\hd_F, \hd_i\}$ on the graph $S_{r_0}(u)$, the coefficients of the second fundamental form is defined by
	$$\hat{\mbox{II}}_{ij}=-\hg\left(\nabla_{\hd_i}\hd_j,\nu(S_{r_0}(u))\right),$$
	where $\nu(S_{r_0}(u))=|\nabla F|\hd_F$ is the unit normal vector to $S_{r_0}(u).$
	
	Following equation \eqref{hg}, we have
	$$\hat{\mbox{II}}_{ij}=-|\nabla F|\hg\left(\hat{\Gamma}_{ij}^k \hd_k, \hd_F\right)=-|\nabla F|\hat{\Gamma}_{ij}^F \hg_{FF}=|\nabla F|^{-1}\mbox{Hess}F\left(\hd_i,\hd_j\right).$$
	
	Therefore, the mean curvature $H(r_0,u,g)$ in coordinates $\{\hd_F, \hd_i\}$, is given by
	\begin{equation}
	\label{1}
	H(r_0,u,g)=|\nabla F|^{-1}\hg^{ij}\mbox{Hess}F\left(\hd_i,\hd_j\right).
	\end{equation}
	
	First, we compute the Hessian in coordinates $\{r,x_i\}$ 
	\begin{equation*}
	\begin{split}
	\mbox{Hess}F\left(\hd_i,\hd_j\right)&=\mbox{Hess}F\left(\d_i \ps\d_r+\d_i,\d_j \ps\d_r+\d_j\right)\\
	&=\d_i \ps\d_j \ps\left[-\Gamma_{00}^0+\Gamma_{00}^k\d_k \ps\right]+\d_i \ps\left[-\Gamma_{0j}^0+\Gamma_{0j}^k\d_k \ps\right]\\
	&\;\;+\d_j \ps\left[-\Gamma_{0i}^0+\Gamma_{0i}^k\d_k \ps\right]+\left[-\Gamma_{ij}^0-\d^2_{ij} \ps+\Gamma_{ij}^k\d_k \ps\right]
	\end{split}
	\end{equation*}
	
	Straightforward computations show that Christoffel symbols are given by
\begin{equation*}
	\begin{split}
	\Gamma_{00}^0&=H_{000}=\mathcal{O} (e^{-r\a}),\\
	\Gamma_{00}^k&=H_{00k}=\mathcal{O} (e^{-r(\a-2)}),\\
	\Gamma_{0j}^0&=H_{0j0}=\mathcal{O} (e^{-r\a}),\\
	\Gamma_{ij}^0&=e^{-2r}\bg_{ij}+H_{ij0}=e^{-2r}\bg_{ij}+\mathcal{O} (e^{-r\a}),\\
	\Gamma_{0i}^k&=-\de_i^k+H_{0ik}=-\de_i^k+\mathcal{O} (e^{-r(\a-2)}),\\
	\Gamma_{ij}^k&=\bar{\Gamma}_{ij}^k+H_{ijk}=\bar{\Gamma}_{ij}^k+\mathcal{O} (e^{-r(\a-2)})
	\end{split}
\end{equation*}
		
	Thus, 
	\begin{equation}
	\begin{split}
	\label{Hess}
	\mbox{Hess}F\left(\hd_i,\hd_j\right)&=-2\d_i \ps\d_j \ps-e^{-2r}\bg_{ij}-\d^2_{ij} \ps+\bar{\Gamma}_{ij}^k\d_k \ps+E^{Hess}_{ij}(r,h,\d_k h,\d_k u,(\d_k u)^2),
	\end{split}
	\end{equation}
	where
	\begin{equation*}
	\begin{split}
	E^{Hess}_{ij}=-H_{000}\d_i \ps\d_j \ps-H_{0j0}\d_i \ps-H_{0i0}\d_j \ps-H_{ij0}+\sum_{k=1}^n(H_{00k}+H_{0jk}+H_{0ik}+H_{ijk})\d_k \ps.
	\end{split}
	\end{equation*}
	
	On the other hand, a direct computation shows that inverse metric satisfies
	$$\hg^{ij}=e^{2r}{\bg}^{ij}+\th^{ij}-\frac{g_{00}(g^{ik}\d_k u)(g^{jl}\d_l u)}{1+|\nabla \ps|^2g_{00}},$$
	where $\th=\mathcal{O} (e^{-r\a}).$ 
	Then, 
	\begin{equation*}
	\begin{split}
	\hg^{ij}
	&=e^{2r}\bg^{ij}+\th^{ij}-\frac{e^{4r}(\bg^{ik}\d_k u)(\bg^{jl}\d_l u)}{1+e^{2r}|\overline{\nabla}u|^2_{\bg}}+e_2^{ij}+e_1^{ij}.
	\end{split}
	\end{equation*}
	where
	\begin{equation*}
	\begin{split}
	e_1^{ij}&=-\left(\frac{e^{4r}h_{00}(\bg^{ik}\d_k u)(\bg^{jl}\d_l u)+[e^{2r}((\bg^{ik}\d_k u)\th^{jl}\d_l u+(\bg^{jl}\d_l u)\th^{ik}\d_k u)+\th^{ik}\d_k u\th^{jl}\d_l u](1+h_{00})}{1+(e^{2r}|\overline{\nabla}u|^2_{\bg}+\th^{ij}\d_i u \d_j u)(1+h_{00})}\right)\\
	e_2^{ij}&=\frac{e^{4r}(\bg^{ik}\d_k u)(\bg^{jl}\d_l u)\left(e^{2r}h_{00}|\overline{\nabla}u|^2_{\bg}+\th^{ij}\d_i u \d_j u(1+h_{00})\right)}{\left(1+e^{2r}|\overline{\nabla}u|^2_{\bg}\right)\left(1+(e^{2r}|\overline{\nabla}u|^2_{\bg}+\th^{ij}\d_i u \d_j u)(1+h_{00})\right)}
	\end{split}
	\end{equation*}
	
	Thus,
	\begin{equation}
	\label{metric}
	\hg^{ij}=e^{2r}\bg^{ij}-\frac{e^{4r}(\bg^{ik}\d_k u)(\bg^{jl}\d_l u)}{1+e^{2r}|\overline{\nabla}u|^2_{\bg}}+\hat{E}^{ij}(r,h,(\d_k u)^2,(\d_k u)^4).
	\end{equation}
	where, $\hat{E}^{ij}=\th^{ij}+e_1^{ij}+e_2^{ij}.$	

	Therefore, using Proposition \ref{F}(ii) and replacing the equations \eqref{Hess} and \eqref{metric} in \eqref{1}, we obtain
	\begin{equation*}
	\begin{split}
	H(r_0,u,g)\sqrt{1+e^{2r}|\overline{\nabla}\ps|^2_{\bg}+E^F}
	&=-e^{2r}\lap_{\bar{g}}\ps+e^{4r}\frac{\hess_{\bg}\ps(\overline{\nabla}\ps,\overline{\nabla}\ps)}{1+e^{2r}|\overline{\nabla}\ps|^2_{\bg}}-\frac{e^{2r}|\overline{\nabla}\ps|^2_{\bg}}{1+e^{2r}|\overline{\nabla}\ps|^2_{\bg}}-n+E^H,
	\end{split}
	\end{equation*}
	where 
	\begin{equation}
	\label{EH}
	E^H=E^H(r,h,\d_k h,\d_k u, \d^2_{ij}u)=\hg^{ij}E^{Hess}_{ij}+\hat{E}^{ij}\mbox{Hess} F(\hd_i,\hd_j)-\hat{E}^{ij}E^{Hess}_{ij}.
	\end{equation}
	
\end{proof}

Notice that in the model case $g=g_N$, the mean curvature of the graph $S_{r_0}(u)$ is just
\begin{equation}
\label{Hm}
\begin{split}
H(r_0,u,g_N)&=\frac{1}{\sqrt{1+e^{2r}|\overline{\nabla}\ps|^2_{\bg}}}\left(-e^{2r}\lap_{\bar{g}}\ps+e^{4r}\frac{\hess_{\bg}\ps(\overline{\nabla}\ps,\overline{\nabla}\ps)}{1+e^{2r}|\overline{\nabla}\ps|^2_{\bg}}-\frac{e^{2r}|\overline{\nabla}\ps|^2_{\bg}}{1+e^{2r}|\overline{\nabla}\ps|^2_{\bg}}-n\right).
\end{split}
\end{equation}

The first important application of Proposition \ref{Hg} is
\begin{Corollary}
	\label{H0}
There exists a constant $\mathcal{C}_2$ such that $$|E^H(u=0)|_{C^{0,\b}(\Sigma)}\leq\mathcal{C}_2 e^{-r_0(\a-2)}.$$
\end{Corollary}

\begin{proof}
Following the equations\eqref{Hess} and \eqref{metric}, we have
	\begin{equation}
	\label{Ehess0}
	E^{Hess}_{ij}(u=0)=-H_{ij0}=\mathcal{O} (e^{-r\a}),
	\end{equation}
and	
	\begin{equation}
	\label{Emetric0}
	\hat{E}^{ij}(u=0)=\th^{ij}=\mathcal{O} (e^{-r\a}).
	\end{equation}
	
	Therefore, Proposition \ref{Hg} implies
\begin{equation*}
\begin{split}
\sqrt{1+E^F(u=0)}\,H(r_0,0,g)&=-n+E^H(u=0)=-n+\mathcal{O} (e^{-r(\a-2)}).
\end{split}
\end{equation*}
\end{proof}


On the other hand, if we consider a normal variation of the slice $\Sigma_{r_0}$, i.e.
$$S_{r_0}(tu)=\{(r_0+tu(x),x)\mid x\in\Sigma \},$$
and we denote by $H(r_0,tu,g)$ its mean curvature. Then

\begin{equation}
\label{taylor}
H(r_0,u,g)=H(r_0,0,g)+\frac{d}{dt}\Big|_{t=0}H(r_0,tu,g)+\frac{d^2}{dt^2}\Big|_{t=0}H(r_0,tu,g).
\end{equation}

\begin{Proposition}
	\label{LemaL}
	$$\sqrt{1+E^F(t=0)}\,\frac{d}{dt}\Big|_{t=0}H(r_0,tu,g)=-e^{2r_0}\lap_{\bg}\ps+E^L,$$
	where $E^L=E^L(r_0, h,\d_k h, u, \d_k u, \d^2_{ij}u)$. Moreover, there exists $R_1>0$ such that if $r_0>R_1$ there exists a constant $\mathcal{C}_3$ such that $$|E^L|_{C^{0,\b}(\Sigma)}\leq\mathcal{C}_3 e^{-r_0(\a-4)}|u|_{C^{2,\beta}(\Sigma)}.$$ 
\end{Proposition}

\begin{proof}
	Using Proposition \ref{Hg}, we obtain that the mean curvature of $S_{r_0}(tu)$ satisfies
	\begin{equation}
	\label{Ht}
	H(r_0,tu,g)f_1(t)\, =\sum_{i=2}^5f_i(t),
	\end{equation}
	where 
\begin{alignat*}{3}
	&f_1(t)=\sqrt{1+e^{2r_t}t^2|\overline{\nabla}\ps|^2_{\bg}+E^F(t)},&\hspace{1em}&
	f_2(t)=-e^{2r_t}t\lap_{\bar{g}}\ps,&\hspace{1em}&
	f_3(t)=e^{4r_t}t^3\frac{\hess_{\bg}\ps(\overline{\nabla}\ps,\overline{\nabla}\ps)}{1+e^{2r_t}t^2|\overline{\nabla}\ps|^2_{\bg}},\\
	&f_4(t)=-\frac{e^{2r_t}t^2|\overline{\nabla}\ps|^2_{\bg}}{1+e^{2r_t}t^2|\overline{\nabla}\ps|^2_{\bg}},&\hspace{1em}&
	f_5(t)=-n+E^H(t),&\hspace{1em}&
	r_t=r_0+t\ps.
\end{alignat*}
	
	Taking derivatives in both sides of the equation \eqref{Ht}, we obtain
	$$f'_1(0)H(r_0,0,g)+f_1(0)\frac{d}{dt}\Big|_{t=0}H(r_0,tu,g)=\sum_{i=2}^5f'_i(0).$$
	
	Thus, it is easy to compute
	\begin{alignat*}{3}
	&f_1(0)=\sqrt{1+E^F(0)}&\hspace{2em}& f'_1(0)=\frac{1}{2\sqrt{1+E^F(0)}}\frac{d}{dt}\Big|_{t=0}E^F&\hspace{2em}&f'_2(0)=-e^{2r_0}\lap_{\bg}\ps\\ &f'_3(0)=f'_4(0)=0&\hspace{2em}&f'_5(0)=\frac{d}{dt}\Big|_{t=0}E^H
	\end{alignat*}
	Therefore, following Corollary \ref{H0} , we have
	$$\sqrt{1+E^F(0)}\frac{d}{dt}\Big|_{t=0}H(r_0,tu,g)=-e^{2r_0}\lap_{\bg}\ps+\frac{d}{dt}\Big|_{t=0}E^H+\frac{n-E^H(0)}{2(1+E^F(0))}\frac{d}{dt}\Big|_{t=0}E^F.$$
	
	Moreover, from Proposition \ref{F} we get that there exists a constant $c_1$ such that
	\begin{equation}
	\label{EF'}
	\left|\frac{d}{dt}\Big|_{t=0}E^F\right|\leq c_1e^{-r(\a-2)}|u|_{C^{2,\beta}(\Sigma)}.
	\end{equation}
	
	Analogously, from Proposition \ref{Hg} it is possible to estimate that the derivative with respect $t$ at $t=0$ of $E^H$ depends on $\{r_0, h,\d_k h, u, \d_k u, \d^2_{ij}u\}.$ Specifically, equations \eqref{Hess}, \eqref{metric} and \eqref{EH} imply 
	\begin{equation*}
	\begin{split}
	\mbox{Hess}F\left(\hd_i,\hd_j\right)(t)&=-2t^2\d_i \ps\d_j \ps-e^{-2r_t}\bg_{ij}-t\d^2_{ij} \ps+t\bar{\Gamma}_{ij}^k\d_k \ps+E^{Hess}_{ij}(t),\\
	\hg^{ij}(t)&=e^{2r_t}\bg^{ij}-\frac{e^{4r_t}t^2(\bg^{ik}\d_k u)(\bg^{jl}\d_l u)}{1+e^{2r_t}t^2|\overline{\nabla}u|^2_{\bg}}+\hat{E}^{ij}(t),\\
	{E^H}(t)&=\hg^{ij}(t)E^{Hess}_{ij}(t)+\hat{E}^{ij}(t)\mbox{Hess} F(\hd_i,\hd_j)(t)-\hat{E}^{ij}(t)E^{Hess}_{ij}(t).
	\end{split}
	\end{equation*}
	
	Thus, using \eqref{Ehess0} and \eqref{Emetric0} we get
	\begin{equation*}
	\begin{split}
	\mbox{Hess}F\left(\hd_i,\hd_j\right)(0)&=-e^{-2r_0}\bg_{ij}+E^{Hess}_{ij}(0)=-e^{-2r_0}\bg_{ij}-H_{ij0}\\
	\hg^{ij}(0)&=e^{2r_0}\bg^{ij}+\hat{E}^{ij}(0)=e^{2r_0}\bg^{ij}+\th^{ij},
	\end{split}
\end{equation*}
	and 
	\begin{equation*}
	\begin{split}
	\frac{d}{dt}\Big|_{t=0}\mbox{Hess}F\left(\hd_i,\hd_j\right)&=2ue^{-2r_0}\bg_{ij}-\d^2_{ij} \ps+\bar{\Gamma}_{ij}^k\d_k \ps+\frac{d}{dt}\Big|_{t=0}{E^{Hess}_{ij}}\\
	\frac{d}{dt}\Big|_{t=0}\hg^{ij\,}&=2ue^{2r_0}\bg^{ij}+\frac{d}{dt}\Big|_{t=0}\hat{E}^{ij\,},
	\end{split}
\end{equation*}
where there exist constants $c_2, \, c_3$ such that
$$\left|\frac{d}{dt}\Big|_{t=0}E^{Hess}_{ij}\right|\leq c_2 e^{-r_0(\a-2)}|u|_{C^{2,\beta}(\Sigma)}, \quad \left|\frac{d}{dt}\Big|_{t=0}\hat{E}^{ij\,}\right|\leq c_3 e^{-r_0\a}|u|_{C^{2,\beta}(\Sigma)}.$$
	
	Therefore,
	\begin{equation*}
	\begin{split}
	\frac{d}{dt}\Big|_{t=0}{E^H}&=e^{2r_0}\bg^{ij}\frac{d}{dt}\Big|_{t=0}{E^{Hess}_{ij}}-2ue^{2r_0}\bg^{ij}H_{ij0}+\th^{ij}(2ue^{-2r_0}\bg_{ij}-\d^2_{ij} \ps+\bar{\Gamma}_{ij}^k\d_k \ps)\\
	&\;-e^{-2r_0}\bg_{ij}\frac{d}{dt}\Big|_{t=0}\hat{E}^{ij\,}+\th^{ij}\frac{d}{dt}\Big|_{t=0}{E^{Hess}_{ij}}-H_{ij0}\frac{d}{dt}\Big|_{t=0}\hat{E}^{ij\,},
	\end{split}
	\end{equation*}
	and then, there exists a constant $c_4$ such that 
	\begin{equation}
	\label{EH'}
	\left|\frac{d}{dt}\Big|_{t=0}{E^H}\right|\leq c_4 e^{-r_0(\a-4)}|u|_{C^{2,\beta}(\Sigma)}.
	\end{equation}
	
	Finally, we get
		$$\sqrt{1+E^F(0)}\frac{d}{dt}\Big|_{t=0}H(r_0,tu,g)=-e^{2r_0}\lap_{\bg}\ps+\underbrace{\frac{d}{dt}\Big|_{t=0}{E^H}+\frac{n-E^H(0)}{2(1+E^F(0))}\frac{d}{dt}\Big|_{t=0}E^F}_{E^L},$$
and using \eqref{EF'} and \eqref{EH'} we get that there exists $R_1>0$ such that if $r_0>R_1$ there exists a constant $\mathcal{C}_3$ such that $|E^L|_{C^{0,\b}(\Sigma)}\leq\mathcal{C}_3 e^{-r_0(\a-4)}|u|_{C^{2,\beta}(\Sigma)}.$
\end{proof}

Now, we need to understand the structure of the quadratic part of the mean curvature operator

\begin{Proposition}
	\label{LemaQ}
	$$\sqrt{1+E^F(t=0)}\,\frac{d^2}{dt^2}\Big|_{t=0}H(r_0,tu,g)=-4e^{2r_0}u\lap_{\bg}\ps-2e^{2r_0}|\overline{\nabla}\ps|^2_{\bg}+\frac{ne^{2r_0}|\overline{\nabla}\ps|^2_{\bg}}{(1+E^F(t=0))}+E^Q,$$
	where $E^Q=E^Q(r_0, h,\d_k h, u,u^2, \d_k u, u\d_k u, (\d_k u)^2,u\d^2_{ij}u,\d_{k}u\d^2_{ij}u)$. Moreover, there exists $R_2>0$ such that if $r_0>R_2$, there exists a constant $\mathcal{C}_4$ such that $$|E^Q|_{C^{0,\b}(\Sigma)}\leq\mathcal{C}_4 e^{-r_0(\a-4)}|u|^2_{C^{2,\beta}(\Sigma)}.$$
\end{Proposition}

\begin{proof}
Analogously to Proposition \ref{LemaL}, with the same notations, we take second derivatives in both sides of the equation \eqref{Ht} 
\begin{equation}
\label{taylor2}
f''_1(0)H(r_0,0,g)+2f'_1(0)\frac{d}{dt}\Big|_{t=0}H(r_0,tu,g)+f_1(0)\frac{d^2}{dt^2}\Big|_{t=0}H(r_0,tu,g)=\sum_{i=2}^5f''_i(0).
\end{equation}

Thus, it is easy to compute
$$f_1(0)=\sqrt{1+E^F(0)},\quad f'_1(0)=\frac{{E^F}'(0)}{2f_1(0)} ,\quad  f''_1(0)=\frac{e^{2r_0}|\overline{\nabla}\ps|^2_{\bg}+{E^F}''(0)}{f_1(0)}-\frac{{E^F}'(0)^2}{2f_1(0)^3},$$
$$f''_2(0)=-4e^{2r_0}u\lap_{\bg}\ps, \quad f''_3(0)=0,\quad f''_4(0)=-2e^{2r_0}|\overline{\nabla}\ps|^2_{\bg},\quad f''_5(0)={E^H}''(0).$$

The formula of this proposition comes from replacing these previous equations in \eqref{taylor2}, where $E^Q$ is given by
	\begin{equation*}
	\begin{split}
	E^Q&=\frac{d^2}{dt^2}\Big|_{t=0}{E^H}-\frac{E^L}{f_1^2(0)}\frac{d}{dt}\Big|_{t=0}E^F+\frac{e^{2r_0}}{f_1^2(0)}\left(\frac{d}{dt}\Big|_{t=0}{E^F}\lap_{\bg}\ps-{E^H}(0)|\overline{\nabla}\ps|^2_{\bg}\right)\\
	&\;\;+\left(\frac{d^2}{dt^2}\Big|_{t=0}{E^F}-\frac{1}{2f_1^2(0)}\left(\frac{d}{dt}\Big|_{t=0}E^F\right)^2\right)\left(\frac{n-{E^H}(0)}{f_1^2(0)}\right)
	\end{split}
	\end{equation*}

Moreover, a straightforward computations similar to Proposition \ref{LemaL} shows that there exist constants $\bc_1, \bc_2$ such that 
\begin{equation}
\label{EH''}
\left|\frac{d^2}{dt^2}\Big|_{t=0}{E^F}\right|\leq \bc_1 e^{-r_0(\a-4)}|u|^2_{C^{2,\beta}(\Sigma)},\quad \left|\frac{d^2}{dt^2}\Big|_{t=0}{E^H}\right|\leq \bc_2 e^{-r_0(\a-4)}|u|^2_{C^{2,\beta}(\Sigma)},
\end{equation}
and then, estimates on $E^Q$ follows.
\end{proof}

Summarizing, given $u\in C^{2,\b}(\Sigma)$, the mean curvature of the graph of $u$ over the slice $\Sigma_{r}$, satisfies
\begin{equation}
\begin{split}
\label{ProperH}
	\sqrt{1+E^F(0)}H(r,0,g)&=-n+E^H(0)\\
	\sqrt{1+E^F(0)}\frac{d}{dt}\Big|_{t=0}H(r,tu,g)&=-e^{2r}\lap_{\bg}\ps+E^L\\
	\sqrt{1+E^F(0)}\frac{d^2}{dt^2}\Big|_{t=0}H(r,tu,g)&=-4e^{2r}u\lap_{\bg}\ps-2e^{2r}|\overline{\nabla}\ps|^2_{\bg}+\frac{ne^{2r}|\overline{\nabla}\ps|^2_{\bg}}{(1+E^F(0))}+E^Q,
	\end{split}
\end{equation}
and there exist constants $\mathcal{C}_1, \mathcal{C}_2,$ such that 
\begin{equation}
\begin{split}
\label{Proper2H}
|E^F(0)|_{C^{0,\b}(\Sigma)}\leq\mathcal{C}_1 e^{-r\a},&\quad |E^H(0)|_{C^{0,\b}(\Sigma)}\leq\mathcal{C}_2 e^{-r(\a-2)},
\end{split}
\end{equation}
and there exists $R'>0$ such that if $r>R'$ then there exists constants $ \mathcal{C}_3,$ $\mathcal{C}_4$ such that

\begin{equation}
\label{Proper3H}
|E^L|_{C^{0,\b}(\Sigma)}\leq\mathcal{C}_3 e^{-r(\a-4)}|u|_{C^{2,\b}(\Sigma)},\quad |E^Q|_{C^{0,\b}(\Sigma)}\leq\mathcal{C}_4 e^{-r(\a-4)}|u|^2_{C^{2,\b}(\Sigma)}.
\end{equation}

\section{Existence of CMC-surfaces}\label{existenceS}
In this section we will prove that, given an end of a $C^{2,\b}$-asymptotically cuspidal manifold $(M,g)$ of order $\a>4$ and given $r$ big enough, there exists a function $u\in C^{2,\beta}(\Sigma)$, depending of $r$, such that the graph of $u$ over the slice $\Sigma_{r}$ has constant mean curvature. To do this, we will use the Taylor's equation for the mean curvature to establish a non-linear partial differential equation and then, to find a solution using the iterative scheme.  

Let $(N,g)$ be an end of a $C^{2,\b}$-asymptotically cuspidal manifold $(M,g)$ of order $\a>4.$ The goal in this section will be to find a function $u(r)\in C^{2,\beta}(\Sigma)$ such that the mean curvature of $S_r(u(r))$ with respect to the metric $g$, $H(r,u(r),g)$, is constant for big $r.$

Following Taylor's equation \eqref{taylor} and \eqref{ProperH} for any positive function $u$, we have that the mean curvature of $S_r(u)$ satisfies
$$\sqrt{1+E^F(0)}H(r,u,g)=-n+E^H(0)-e^{2r}\lap_{\bg}\ps+E^L+e^{2r}Q(u)+E^Q,$$
where the quadratic part $Q(u):=-4u\lap_{\bg}\ps-2|\overline{\nabla}\ps|^2_{\bg}+\frac{n|\overline{\nabla}\ps|^2_{\bg}}{(1+E^F(0))}.$ 

To solve the equation $H(r,u,g)=-n+\de$, for some constant $\de$, is equivalent to solve the non-linear partial differential equation
\begin{equation}
	\label{EQ}
	\begin{split}
	\lap_{\bg}\ps\!&=\!e^{-2r}\!\!\left[n(\sqrt{1+E^F(0)}-1)+E^H(0)-\de \sqrt{1+E^F(0)}\right]\!\!+\!e^{-2r}E^L(u)\!+\!Q(u)\!+\!e^{-2r}E^Q(u).
	\end{split}
\end{equation}

Since $\lap_{\bg}:C^{2,\beta}(\Sigma) \to C^{0,\beta}(\Sigma)$ defines an elliptic and continuous linear differential operator, the kernel $Ker(\lap_{\bg})\subset C^{2,\beta}(\Sigma)$ is closed. Moreover, compactness of $\Sigma$ implies that $C^{2,\beta}(\Sigma) \hookrightarrow L^2(\Sigma)$ and $Ker(\lap_{\bg})$ consists on constant functions. 

Therefore, the projection on the kernel exists, is continuous and for every $u\in C^{2,\beta}(\Sigma)$ it holds
$$u=u^{\perp}+\frac{1}{|\Sigma|}\int_{\Sigma} u.$$

The idea to find a solution of \eqref{EQ}, is to prove that the sequence $\{u_j^\perp\} $ in $C^{2,\beta}(\Sigma)$ defined by the iterative method
\begin{equation}
\label{iterative}
\begin{split}
u_0^\perp&=0\\
\lap_{\bg}\ps_{j+1}^\perp&=e^{-2r}\left[n(\sqrt{1+E^F(0)}-1)+E^H(0)-\de_j \sqrt{1+E^F(0)}+E^L(u_j^\perp)+e^{2r}Q(u_j^\perp)+E^Q(u_j^\perp)\right]
\end{split}
\end{equation}
with $\{\de_j\}$ given by
\begin{equation}
\label{delta}
\begin{split}
\de_0:=&\left(\int_{\Sigma}\sqrt{1+E^F(0)}\right)^{-1}\left(\int_{\Sigma}n(\sqrt{1+E^F(0)}-1)+E^H(0)\right),\\
\de_j:=& \,\de_0+\left(\int_{\Sigma}\sqrt{1+E^F(0)}\right)^{-1}\left(\int_{\Sigma}E^L(u_j^\perp)+e^{2r}Q(u_j^\perp)+E^Q(u_j^\perp)\right),
\end{split}
\end{equation}
converges.

\begin{Remark}
	Note that the definition of $\{\de_j\}$ is to ensure that the iterative scheme is well-defined, in the sense that $\lap_{\bg}\ps_{j+1}^\perp$ in $C^{0,\beta}(\Sigma)$ to be zero mean value functions.
\end{Remark}

To show convergence of the iterative scheme describe in \eqref{iterative}, it is necessary two proper estimates for the inverse of $\lap_{\bg}$ and for the quadratic part $Q.$ The proof of the following is directly adapted from an analogous result in \cite{L}.

\begin{Lemma}
	\label{BoundL}
	Let $\lap_{\bg}:C^{2,\b}(\Sigma)\to C^{0,\b}(\Sigma)$ be the Laplace-Beltrami operator.	If $u\in C^{2,\b}(\Sigma)$ and $\lap_{\bg}u=f$, then there exists a constant $C_L>0$ such that
	$$\left|u-\frac{1}{|\Sigma|}\int_{\Sigma} u\right|_{C^{2,\b}(\Sigma)}\leq C_L|f|_{C^{0,\b}(\Sigma)},$$
	where $C_L$ is a constant independent of $u$ and $f$.
\end{Lemma}

\begin{proof}
	Suppose by contradiction that there exist sequence $\{u_j\}$ such that
	$$ \lap_{\bg} u_j=f_j,\quad \left|u_j-\frac{1}{|\Sigma|}\int_{\Sigma} u_j\right|_{C^{2,\b}(\Sigma)}=1,\quad |f_j|_{C^{0,\b}(\Sigma)}\to 0.$$
	
	Then, up to subsequence, $\left\{u_j-\frac{1}{|\Sigma|}\int_{\Sigma} u_j\right\}$ converges to $v$ in $C^2(\Sigma).$ So, taking the point-wise limit we get
	$$\lim_{j\to \infty}\lap_{\bg}\left(u_j-\frac{1}{|\Sigma|}\int_{\Sigma} u_j\right)=\lap_{\bg}v.$$
	And also, we have
	\begin{equation*}
	\lim_{j\to \infty}\lap_{\bg}\left(u_j-\frac{1}{|\Sigma|}\int_{\Sigma} u_j\right)=\lim_{j\to\infty} f_j=0.
	\end{equation*}
	Therefore, $v\in Ker(\lap_{\bg})$, i.e. $v$ is a constant function. Moreover, using that $\left|u_j-\frac{1}{|\Sigma|}\int_{\Sigma} u_j\right|_{C^{2,\b}(\Sigma)}\!=1$ and compactness of $\Sigma$, Dominated convergence Theorem implies
	
	\begin{equation*}
	v=\frac{1}{|\Sigma|}\int_{\Sigma}v=\frac{1}{|\Sigma|}\int_{\Sigma}\left(\lim_{j\to\infty}\left(u_j-\frac{1}{|\Sigma|}\int_{\Sigma}u_j\right)\right)=\frac{1}{|\Sigma|}\lim_{j\to\infty}\left(\int_{\Sigma}\left(u_j-\frac{1}{|\Sigma|}\int_{\Sigma}u_j\right)\right)=0.
	\end{equation*}
	Therefore, up to a subsequence, $\left\{u_j-\frac{1}{|\Sigma|}\int_{\Sigma} u_j\right\}$ converges to $0$ in $C^2(\Sigma).$
	
	On the other hand,  interior Schauder estimates for $\lap_{\bg}$ implies that there exists a constant $C>0$, depending on $n, \b$, such that
	$$|u|_{C^{2,\b}(\Omega_k)}\leq C(n,\b)\left(|u|_{C^{0}(\Omega'_k)}+|\lap_{\bg}u|_{C^{0,\b}(\Omega'_k)}\right),$$
	where $\{\Omega_k\subset \Omega'_k\}_k$ is a finite covering of concentric open balls of $\Sigma.$	Therefore,
	\begin{equation*}
	\begin{split}
	\left|u_j-\frac{1}{|\Sigma|}\int_{\Sigma} u_j\right|_{C^{2,\b}(\Omega_k)}
	&\leq C(n,\b)\left(\left|u_j-\frac{1}{|\Sigma|}\int_{\Sigma} u_j\right|_{C^{0}(\Omega'_k)}+|\lap_{\bg}u_j|_{C^{0,\b}(\Omega'_k)}\right).
	\end{split}
	\end{equation*}
	
	Taking the supremum over the covering, we get
	\begin{equation*}
	\begin{split}
	1=\sup_{\{\Omega_k\}}\left|u_j-\frac{1}{|\Sigma|}\int_{\Sigma} u_j\right|_{C^{2,\b}(\Omega_k)}&\leq C(n,\b)\left(\left|u_j-\frac{1}{|\Sigma|}\int_{\Sigma} u_j\right|_{C^{0}(\Sigma)}+|f_j|_{C^{0,\b}(\Sigma)}\right)
	\end{split}
	\end{equation*}
	but, since all the terms on the right side tend to zero as $j\to\infty$, this is a contradiction.
\end{proof}

We can now prove the main result of this section which proves Theorem \ref{Main}(i)(ii).
	
\begin{Theorem}
	\label{ThCMCsurface}
Let $(N,g)$ be an end of a $C^{2,\b}$-asymptotically cuspidal manifold $(M,g)$ of order $\a>4$. Then, there exists $R>0$ such that for every $r>R$ there exists a function $u(r)\in C^{2,\b}(\Sigma)$ with zero mean value such that
$$S_r(u(r))=\{(r+u(r)(x))\mid x\in \Sigma\}\subset N$$
has constant mean curvature, equal to $(-n+\de)$, for some constant $\de$.

Moreover, there exist constants $C(R)>0$ and $c(R)>0$ such that $$|u(r)|_{C^{2,\b}(\Sigma)}\leq Ce^{-r(\a-2)},\quad |\de|\leq ce^{-r(\a-4)} .$$
\end{Theorem}

\begin{proof}
	Following the description in this section, given one end $(N,g)$ in a $C^{2,\b}$-asymptotically cuspidal manifold, we define, for any positive number $r$ and positive function $u$, the hypersurface $S_r(u)$. Since to solve $H(r,u,g)=-n+\de$, for some constant $\de$, is equivalent to solve \eqref{EQ}, we will consider the sequence $\{u_j^\perp\} $ in $C^{2,\beta}(\Sigma)$ defined by the iterative method \eqref{iterative} given by
	\begin{equation*}
	\begin{split}
	u_0^\perp&=0\\
	\lap_{\bg}\ps_{j+1}^\perp&=e^{-2r}\left[n(\sqrt{1+E^F(0)}-1)+E^H(0)-\de_j \sqrt{1+E^F(0)}+E^L(u_j^\perp)+e^{2r}Q(u_j^\perp)+E^Q(u_j^\perp)\right],
	\end{split}
	\end{equation*}
	where $\de_j$ is defined in \eqref{delta}.
	
	First, we will estimate $\de_j$. We have that
	\begin{equation*}
	\begin{split}
	|\de_0|=&\left(\int_{\Sigma}\sqrt{1+E^F(0)}\right)^{-1}\left|\int_{\Sigma}n(\sqrt{1+E^F(0)}-1)+E^H(0)\right|\\
	\leq&\left(\int_{\Sigma}\sqrt{1+E^F(0)}\right)^{-1}|\Sigma|\left|n(\sqrt{1+E^F(0)}-1)+E^H(0)\right|_{C^{0,\b}(\Sigma)}.
	\end{split}
	\end{equation*}
	
	Then, using \eqref{Proper2H}, there exists $R'''$ such that for every $r>R'''$ there exists a constant $c_0$ such that
	\begin{equation}
	\label{EstDe}
	|\de_0|\leq\,c_0e^{-r\a}.
	\end{equation}
	
	Moreover, we have that
	\begin{equation*}
	\begin{split}
	|\de_j|\leq&\, |\de_0|+\left(\int_{\Sigma}\sqrt{1+E^F(0)}\right)^{-1}\left|\int_{\Sigma}E^L(u_j^\perp)+e^{2r}Q(u_j^\perp)+E^Q(u_j^\perp)\right|\\
	\leq&\, |\de_0|+\left(\int_{\Sigma}\sqrt{1+E^F(0)}\right)^{-1}|\Sigma|\left|E^L(u_j^\perp)+e^{2r}Q(u_j^\perp)+E^Q(u_j^\perp)\right|_{C^{0,\b}(\Sigma)},
	\end{split}
	\end{equation*}
	
	and by definition of $Q$, we get $$|Q(u)|_{C^{0,\b}(\Sigma)}\leq \left(4+\left|\frac{n}{1+E^F(0)}-2\right|\right)|u|^2_{C^{2,\b}(\Sigma)}.$$
	Then, there exists $R''>0$ such that for every $r>R''$ there exists a constant $C_Q>0$, independent of $r$ and $u$, such that
	\begin{equation}
	\label{BoundQ}
	\left|Q(u)\right|_{C^{0,\b}(\Sigma)}\leq C_Q|u|^2_{C^{2,\b}(\Sigma)}.
	\end{equation}
	
	Therefore, using estimates \eqref{Proper3H}, \eqref{EstDe} and \eqref{BoundQ}, there exists $R=\max\{R', R'', R'''\}$ such that for every $r>R$ the following holds
	\begin{equation}
	\label{EstDej}
	|\de_j|\leq\,c_0e^{-r\a}+\mathcal{C}_3 e^{-r(\a-4)}|u_j^\perp|_{C^{2,\b}(\Sigma)}+e^{2r}C_Q|u_j^\perp|^2_{C^{2,\b}(\Sigma)}+\mathcal{C}_4 e^{-r(\a-4)}|u_j^\perp|^2_{C^{2,\b}(\Sigma)}.
	\end{equation}

	Now, we will estimate $u_j^\perp$. Following the iterative scheme \eqref{iterative}, we have that $u_1^\perp$ is defined by 
	\begin{equation*}
	\begin{split}
	\lap_{\bg}\ps_1=\lap_{\bg}\ps_1^\perp=e^{-2r}\left[n(\sqrt{1+E^F(0)}-1)+E^H(0)-\de_0 \sqrt{1+E^F(0)}\right].
	\end{split}
	\end{equation*}
	
	Then, using Lemma \ref{BoundL} and estimates \eqref{Proper2H} and \eqref{EstDe} we get that for every $r>R'''$ there exists a constant $C_1$ such that
	\begin{equation*}
	\begin{split}
	|u_1^\perp|_{C^{2,\b}(\Sigma)}&\leq C_L e^{-2r}\left|n(\sqrt{1+E^F(0)}-1)+E^H(0)-\de_0 \sqrt{1+E^F(0)}\right|_{C^{0,\b}(\Sigma)}\leq C_1e^{-r\a}.
	\end{split}
	\end{equation*}
	
	Moreover, this and \eqref{EstDej} imply that there exists a constant $c_1$ such that 
	$$|\de_1|\leq c_1e^{-r(\a-4)}.$$
	
	Now, $u_2^\perp$ is defined by 
	\begin{equation*}
	\begin{split}
	\lap_{\bg}\ps_2^\perp=e^{-2r}\left[n(\sqrt{1+E^F(0)}-1)+E^H(0)-\de_1 \sqrt{1+E^F(0)}+E^L(u_1^\perp)+e^{2r}Q(u_1^\perp)+E^Q(u_1^\perp)\right].
	\end{split}
	\end{equation*}
	
	Then, using Lemma \ref{BoundL}, inequalities \eqref{Proper2H}, \eqref{Proper3H}, \eqref{BoundQ} and \eqref{EstDe} and the previous estimates for $u_1^\perp$ and $\de_1$, we obtain that for every $r>R$ there exists a constant $C_2$ such that
	\begin{equation*}
	\begin{split}
	|u_2^\perp|_{C^{2,\b}(\Sigma)}
	&\leq C_2e^{-r(\a-2)}.
	\end{split}
	\end{equation*}
	
	Moreover, this and \eqref{EstDej} imply that there exists a constant $c_2$ such that 
	$$|\de_2|\leq c_2e^{-r(\a-4)}.$$
	
	Now, $u_{j+1}^\perp$ is given by
	\begin{equation*}
	\begin{split}
	\lap_{\bg}\ps_{j+1}^\perp=e^{-2r}\left[n(\sqrt{1+E^F(0)}-1)+E^H(0)-\de_j \sqrt{1+E^F(0)}+E^L(u_j^\perp)+e^{2r}Q(u_j^\perp)+E^Q(u_j^\perp)\right].
	\end{split}
	\end{equation*}
	
	Analogously the previous computations, using Lemma \ref{BoundL} and estimates \eqref{Proper2H}, \eqref{Proper3H}, \eqref{BoundQ} and \eqref{EstDe}, we obtain that for every $r>R$ is satisfied
	\begin{equation*}
	\begin{split}
	|u_{j+1}^\perp|_{C^{2,\b}(\Sigma)}
	&\leq C_L e^{-2r}\left|n(\sqrt{1+E^F(0)}-1)+E^H(0)\right|_{C^{0,\b}(\Sigma)}+C_L e^{-2r}|\de_j| \left|\sqrt{1+E^F(0)}\right|_{C^{0,\b}(\Sigma)}\\
	&\;\;+C_L\left(\mathcal{C}_3e^{-r(\a-2)}|u_j^\perp|_{C^{2,\b}(\Sigma)}+C_Q|u_j^\perp|_{C^{2,\b}(\Sigma)}^2+\mathcal{C}_4e^{-r(\a-2)}|u_j^\perp|_{C^{2,\b}(\Sigma)}^2\right).
	\end{split}
	\end{equation*}
	
	If we assume that for every $r>R$ there exists a constant $C_j$ such that $|u_j^\perp|_{C^{2,\b}(\Sigma)}\leq C_je^{-r(\a-2)}$ then \eqref{EstDej} implies that there exists 
	a constant $c_j$ such that $|\de_j|\leq c_je^{-r(\a-4)}.$ Thus, using the previous estimate for $u_{j+1}^\perp$, there exists a constant $C_{j+1}$ such that
	$$|u_{j+1}^\perp|_{C^{2,\b}(\Sigma)}\leq C_{j+1}e^{-r(\a-2)}.$$
	
	Therefore, by induction, the sequence $\{u_j^\perp\}$ is uniformly bounded. Since $u_j\in C^{2,\b}(\Sigma)$, Ascoli-Arzela Theorem implies that, up to a subsequence, $\{u_j^\perp\}$ converges to $u^\perp$ in $C^2(\Sigma).$ Thus, the sequence $\{\de_j\}$ converges to $\de$, defined by
	$$\de:= \,\de_0+\left(\int_{\Sigma}\sqrt{1+E^F(0)}\right)^{-1}\left(\int_{\Sigma}E^L(u^\perp)+e^{2r}Q(u^\perp)+E^Q(u^\perp)\right)$$
	
	Moreover, $u^\perp$ is a solution of
	$$\lap_{\bg}\ps^\perp=e^{-2r}\left[n(\sqrt{1+E^F(0)}-1)+E^H(0)-\de \sqrt{1+E^F(0)}+E^L(u^\perp)+e^{2r}Q(u^\perp)+E^Q(u^\perp)\right],$$
and there exists $R>0$ such that for every $r>R$ there exist constants $C>0$ and $c>0$ such that $$|u^\perp|_{C^{2,\b}(\Sigma)}\leq Ce^{-r(\a-2)},\quad |\de|\leq ce^{-r(\a-4)}.$$
 	
Therefore, $S_r(u^\perp)$ has constant mean curvature equal to $-n+\de$, with respect to the metric $g$. This concludes the proof.
\end{proof}

\section{Existence of a strongly stable CMC-foliation}\label{foliation}
In this section, following the implicit function approach in \cite{Am}, we will prove the existence of strongly stable constant mean curvature surfaces at each end $(N,g)$ near infinity. 

\begin{Definition}
	Let $\mathcal{G}(N)$ be the set of metrics $g$ on $N$ such that 
	$$|(g-g_N)_{ij}|_{C^{2,\b}(N)}\leq C e^{-r\a},\quad \a>4.$$
	 Moreover, the space $\mathcal{G}(N)$ has a distance function 
	$$d(g_1,g_2):=\sup_{(r,x)\in N}\left\{e^{r\a}|(g_1-g_2)_{ij}|_{C^{2,\b}(N)}\right\}$$
\end{Definition}

\begin{Definition}
	A constant mean curvature surface $S$ in $(N,g)$ is called strongly stable if its Jacobi operator
	$$J_S:=\lap_S+Ric(\nu,\nu)+|A|^2,$$
	is such that $$-\int_S J_S(\phi)\phi dS > 0,\quad \forall \phi \in C^{\infty}(S),\; \mbox{with}\; \int_S \phi dS=0.$$
	Here $Ric$ is the Ricci curvature of $(N,g)$, $\nu$ is the unit normal vector to $S$ and $A$ is the second fundamental form of $S.$ 
\end{Definition}

Thus, the main result in this section, which proves Theorem\ref{Main}(i)(iii), is the following

\begin{Theorem}
	\label{ThCMCfoliation}
Given $R>0$ there exists $\e_R>0$ and $\eta_0>0$ such that for all positive $\eta<\eta_0$ the following property holds: For every $g\in \mathcal{G}(N)$ with $d(g,g_N)<\e_R$ and for every $r>R$, there exists a unique function $u(r,g)\in C^{2,\b}(\Sigma)$ with 
$$|u(r,g)|_{C^{2,\b}(\Sigma)}<\eta,\quad \int_{\Sigma} u=0,$$ such that $S_{r}(u(r,g))$ has constant mean curvature with respect to the metric $g.$

Moreover, $\{S_{r}(u(r,g))\}_{r\in (R,\infty)}$ gives a foliation in a region of $(N,g)$ by strongly stable constant mean curvature surfaces.
\end{Theorem}

\begin{proof}
	Consider the spaces
	$$E:=\{u\in C^{2,\b}(\Sigma)\mid \int_{\Sigma} u=0\},\quad F:=\{v\in C^{0,\b}(\Sigma)\mid \int_{\Sigma} v=0\}.$$
	Given $R>0$, we define the $C^1$-map
	$$\Psi: (R,\infty)\times C^{2,\b}(\Sigma)\times \mathcal{G}(N)\to F,\quad \Psi(r,u,g)=H(r,u,g)-\frac{1}{|\Sigma|}\int_{\Sigma}H(r,u,g),$$
	where $H(r,u,g)$ is the mean curvature of $S_r(u)$ with respect to the metric $g.$
	
	Note that $\Psi(r,u,g)=0$ if and only if $H(r,u,g)$ is constant. Therefore, $\Psi(r,0,g_N)=0.$
	
	Moreover, Lemma \ref{LemaL} implies that the differential map $D_u\Psi(r,0,g_N): C^{2,\b}(\Sigma)\to F$ is defined by
	$$D_u\Psi(r,0,g_N)=\frac{d}{dt}\Big|_{t=0}H(r,tu,g_N)-\frac{1}{|\Sigma|}\int_{\Sigma}\frac{d}{dt}\Big|_{t=0}H(r,tu,g_N)=-e^{2r}\lap_{\bg}\ps,$$
	and then, $D_u\Psi(r,0,g_N)$ restricted to $E$ is an isomorphism. 
	Therefore, Implicit Function Theorem implies that there exists $\e_R>0$ and $\eta_0>0$ such that for all positive $\eta<\eta_0$, for every $r>R$ and every $g\in \mathcal{G}(N)$ with $d(g,g_N)<\e_R$ there exists a $C^1$-function $$(r,g)\to u(r,g)\in E,\quad \mbox{with }\, |u(r,g)|_{C^{2,\b}(\Sigma)}<\eta,\quad u(r,g_N)=0,$$ uniquely defined by $\Psi(r,u(r,g),g)=0.$
	
	Thus, we construct a family of surfaces $\{S_r(u(r,g))\}$, where $u(r,g)\in E$ with constant mean curvature with respect to the metric $g.$ To prove that family is a foliation, it is necessary to prove that $\d_r(r+u(r,g))>0.$ 
	
	Since, for $g_N$ the function $u(r,g_N)\equiv0$ then $\d_r(r+u(r,g_N))>0.$ Hence, by continuity, \linebreak $\d_r(r+u(r,g))$ is also positive for all metrics $g\in\mathcal{G}(N)$ with $d(g,g_N)<\e_R$, at least when we choose a possibly smaller $\e_R.$
	
	Finally, to prove that the CMC-surfaces in the family $\{S_r(u(r,g))\}$ are strongly stable, it is necessary to study their Jacobi operator $J$. 
	
	For the CMC-surfaces $S_r(u(r,g_N))=S_r(0)$ we have that their Jacobi operator are given by
	$$J_{S_r(0)}:=\lap_{S_r(0)}+Ric(\d_r,\d_r)+|A|^2=e^{2r}\lap_{\bg},$$ thus, for every $\phi \in C^{\infty}(S_r(0)),\; \mbox{with}\; \int_{S_r(0)} \phi=0$ is satisfied
	$$-\int_{S_r(0)} J_{S_r(0)}(\phi)\phi=-\int_{S_r(0)}e^{2r}\phi\lap_{\bg}\phi=e^{2r}\int_{S_r(0)}|\overline{\nabla}\phi|^2_{\bg} > 0.$$
	Therefore, the foliation $\{S_r(u(r,g_N))\}$ consists in strongly stable CMC-surfaces. And also, by continuity of $u$ and arguing by contradiction, we conclude that for a possibly smaller $\e_R$ such that for every $g\in\mathcal{G}(N)$ with $d(g,g_N)<\e_R$ the Jacobi operator $J_{S_r(u(r,g))}$ satisfies
	$$-\int_{S_r(u(r,g))} J_{S_r(u(r,g))}(\phi)\phi > 0,\quad \forall \phi \in C^{\infty}({S_r(u(r,g))}),\; \mbox{with}\; \int_{S_r(u(r,g))} \phi=0.$$
\end{proof}

\section{Isoperimetric profile}\label{isoperimetric}
In this section we will prove Theorem \ref{Main}(iv).

\begin{Definition}
Let $(M,g)$ be a Riemannian manifold with finite volume equal to $V(M).$ We define the isoperimetric profile as a function
$$I:(0,V(M))\to \R,\quad I(v):=\min\{A(\d\Omega_v)\mid \Omega_v \subset M, V(\Omega_v)=v\},$$
where $A(\d \Omega_v)$ is the area of the boundary of $\Omega_v$ and $V(\Omega_v)$ is the volume of $\Omega_v.$

Moreover, we say that $\Omega_v$ is an isoperimetric region with volume $v$ if $I(v)=A(\d\Omega_v)$.
\end{Definition}

Since an asymptotically cuspidal manifold $(N,g)$, according to Definition \ref{ACdef}, is a complete non-compact manifold with finite volume equal to $V(N)$, \cite[Theorem 2.1]{RR} implies existence of isoperimetric regions for every volume and also that its boundary has constant mean curvature. Moreover, every region $S$ such that its boundary $\d S$ has constant mean curvature $\lambda$ is a critical point of the functional $A(\d S)+\lambda V(S)$ (see \cite{P}). 

Therefore, if $H_v$ denotes the mean curvature of $\d \Omega_v$ then $\Omega_v$ is a critical point of 
\begin{equation}
\label{critical}
A(\d \Omega_v)+H_v V(\Omega_v).
\end{equation}

Thanks to \cite[Theorem 18(c)]{Ros} and \cite{N}, it is enough to prove the following generalization of \cite{AM} for the negatively curved exact model.
\begin{Theorem}
	\label{Thisoperimetric}
Let $(N,g)$ be an end of an asymptotically cuspidal manifold with finite volume equal to $V(N).$ There exists $\e>0$ (small) such that for every $v<\e$ there exists $r_v$ and $u(r_v,g)\in C^{2,\b}(\Sigma)$ with zero mean value such that the boundary of the isoperimetric region $\Omega_v$ coincides with $S_{r_v}(u(r_v,g)).$
\end{Theorem}

\begin{proof}
Since the slices $\Sigma_{r}$ have constant mean curvature equal to $-n$, with respect to the metric $g_N$, we will prove that the mean curvature of the boundary of a isoperimetric region, with respect to the metric $g$, is bounded by $n$. More precisely, for every sequence $\{v_j\}$ which converges to zero there exists a sequence of isoperimetric regions $\{\Omega_{v_j}\}$ which minimizes the functional $A-nV$ then the mean curvature of the boundary of $\Omega_{v_j}$, with respect to the metric $g$, satisfies $H_{v_j}\geq -n.$ Arguing by contradiction, let $\{v_j\}$ be a sequence which converges to zero, then any isoperimetric regions $\{\Omega_{v_j}\}$ which minimize the functional $A-nV$ satisfy $H_{v_j}< -n.$ Moreover, using \eqref{critical} we have
$$0\leq A'(\d \Omega_{v_j})-nV'(\Omega_{v_j})=-H_{v_j}V'(\Omega_{v_j})-nV'(\Omega_{v_j})=-(H_{v_j}+n)V'(\Omega_{v_j}).$$
Since, $\{v_j\}$ converges to zero, then $(H_{v_j}+n)\geq 0.$ This is a contradiction with $H_{v_j}< -n.$

Therefore, if we choose always the normal vector such that the mean curvature in non-positive, we can choose $\e>0$ such that for every $v<\e$ there exists an isoperimetric region $\Omega_v$ which is a minimizer of $A-nV$ and $|H_v|\leq n.$

Moreover, since the injectivity radius of $(N,g)$ at $p\in \d\Omega_v$ is bounded below, a standard covering argument using disjoint geodesic spheres of small volume, shows, throught the monotonicity formula for submanifolds, that for every $v<\e$ there exists $r_v$ such that the isoperimetric region $\Omega_v$ is contained in $[r_v,\infty)\times \Sigma.$

Now, we move the constant mean curvature surface $S_{r_v}(u(r_v,g))$, constructed in Theorem \ref{ThCMCsurface}, until the first contact with $\d\Omega_v$. Let $p$ be the contact point, and $B_{\rho}(p)$ be a geodesic spheres of small volume near $p$, then we may represent both hypersurfaces, inside the spheres, by graphs of functions $f_1$ and $f_2$, respectively. Applying the maximum principle for the constant mean curvature equation, we get
$$S_{r_v}(u(r_v,g))\cap B_{\rho}(p)=\d\Omega_v\cap B_{\rho}(p),\quad \mbox{for some small }\, \rho.$$
Since $f_1$ and $f_2$ satisfy the elliptic PDE for constant mean curvature, $f_1$ and $f_2$ are analytic. Therefore, $S_{r_v}(u(r_v,g))$ coincides with $\d\Omega_v.$
\end{proof}

\bigskip
\bigskip
\medskip
\begin{flushright}
	{\it Abdus Salam's International Centre for Theoretical Physics, \\
	ICTP\\
		 Strada Costiera 11, 34151, Trieste, Italy.\\}
		  arezzo@ictp.it, \hspace{.3mm} kcorrale@ictp.it
\end{flushright}
\end{document}